\documentclass{siamart190516}
\usepackage[normalem]{ ulem }
\usepackage{soul}
\usepackage[utf8]{inputenc}
\usepackage[T1]{fontenc}
\usepackage{amssymb}

\newtheorem{remark}[theorem]{Remark}

\newtheorem{assumption}[theorem]{Assumption}
\newcommand{\realcertify}{\mathtt{RealCertify}}
\DeclareMathOperator{\trace}{Trace}

\newcommand{\N}{\mathbb{N}}
\newcommand{\R}{\mathbb{R}}
\def\M{\mathbf{M}}
\def\B{\mathbf{B}}

\def\P{\mathbf{P}}
\def\S{\mathbf{S}}
\def\C{\mathbf{C}}
\def\b{\mathbf{b}}
\def\c{\mathbf{c}}

\def\X{\mathbf{X}}
\def\Y{\mathbf{Y}}

\def\Q{\mathcal{Q}}
\def\I{\mathbf{I}}
\def\D{\mathbf{D}}

\def\K{\mathbf{K}}
\def\F{\mathbf{F}}

\def\x{\mathbf{x}}
\def\y{\mathbf{y}}

\begin{document}


\title{In SDP relaxations, inaccurate solvers do robust optimization\thanks{Submitted to the editors DATE.
\funding{The research of the first author was funded by the European Research Council (ERC) under the European's Union Horizon 2020 research and innovation program (grant agreement 666981 TAMING project).
The second author benefited from the support of the FMJH Program PGMO (EPICS project) and  EDF, Thales, Orange et Criteo, as well as from the Tremplin ERC Stg Grant ANR-18-ERC2-0004-01 (T-COPS project).
}}}

\author{Jean-Bernard Lasserre\thanks{CNRS; LAAS; 7 avenue du colonel Roche, F-31400 Toulouse; France
  (\email{lasserre@laas.fr}, \url{https://homepages.laas.fr/lasserre/}).} \and Victor Magron\thanks{CNRS; LAAS; 7 avenue du colonel Roche, F-31400 Toulouse; France 
  (\email{vmagron@laas.fr}, \url{https://homepages.laas.fr/vmagron/}).}}
  
\maketitle

\begin{abstract}
We interpret some wrong results (due to 
numerical inaccuracies) already observed when solving SDP-relaxations for polynomial optimization on a double precision floating point
SDP solver. It turns out that this behavior can be explained and justified satisfactorily by a relatively simple paradigm. 
In such a situation, the  SDP solver, and not the user, performs some ``robust optimization'' without
being told to do so. Instead of solving the original optimization problem with nominal criterion $f$, it
uses a new criterion $\tilde{f}$ which belongs to a 
ball $\B_\infty(f,\varepsilon)$ of small radius $\varepsilon>0$, centered at the nominal criterion $f$ 
in the parameter space. In other words the resulting procedure can be viewed as a ``$\max-\min$'' robust optimization
problem with two players (the solver which maximizes on $\B_\infty(f,\varepsilon)$ and the user who minimizes over
the original decision variables). A mathematical rationale behind this ``autonomous'' behavior is described.
\end{abstract}



\section{Introduction}
Certified optimization algorithms provide a way to ensure the safety of several systems in engineering sciences, program analysis as well as cyber-physical critical components. Since these systems often involve nonlinear functions, such as polynomials, it is highly desirable to design certified polynomial optimization schemes and to be able to interpret the behaviors of numerical solvers implementing these schemes.
Wrong results (due to numerical inaccuracies) in some output results from semidefinite programming (SDP) solvers have been observed in quite different applications, and notably in recent applications of the Moment-SOS hierarchy for solving
polynomial optimization problems, see~e.g.,~\cite{Waki12,Waki2012}. In fact this particular application 
has even become a source of illustrating examples for potential pathological behavior of SDP solvers~\cite{Pataki}.
An intuitive mathematical rationale for the wrong results has been already provided informally in~\cite{SOSApprox07} and~\cite{Navascues13}, but does not yield a satisfactory picture for the whole process.

An immediate and irrefutable negative conclusion is that double precision floating point SDP solvers are not robust 
and cannot be trusted as they sometimes provide wrong results in these so-called ``pathological'' cases.
The present paper (with a voluntarily provocative title) is an attempt to provide a different and more positive viewpoint around 
the interpretation of such inaccuracies in SDP solvers, at least when applying the Moment-SOS hierarchy 
of semidefinite relaxations in polynomial optimization as described in \cite{Las01sos,lasserre2009moments}.

We claim that in such a situation, in fact the floating point SDP solver, and not the user, is precisely doing some {\rm robust optimization}, 
{\em without being told to do so}. It solves a ``$\max-\min$" problem in a two-player zero-sum game where the solver is the leader who maximizes (over 
some ball of radius $\varepsilon>0$) in the parameter space of the criterion, and the user is a ``follower''
who minimizes over the original decision variables. In traditional robust optimization, one solves the ``$\min-\max$"
problem where the user (now the leader) minimizes to find a ``robust decision variable'', whereas the SDP solver (now the follower) maximizes in the same ball of the parameter space. In this convex relaxation case, both $\min-\max$ and $\max-\min$ problems give the same solution. So it is fair to say that {\em the solver} is doing what the optimizer should have done in robust optimization.

As an active (and even leader) player of this game, the floating point SDP solver can also play with its two 
parameters which are (a) the threshold level for eigenvalues to declare a matrix positive semidefinite, and (b)
the tolerance level at which to declare a linear  equality constraint to be satisfied. Indeed, the result 
of the ``$\max-\min$" game strongly depends 
on the absolute value of both levels, as well as on their relative values.

{Of course and so far, the rationale behind this viewpoint} which provides a more positive view of inaccurate results from semidefinite solvers,
is proper to the context of semidefinite relaxations for polynomial optimization. Indeed in such a context
we can exploit a mathematical rationale to explain and support this view. 
{An interesting issue is to validate this viewpoint} to a larger class 
of semidefinite programs and perhaps the canonical form of SDPs:
\[\min_\X\,\{\,  \langle \F_0, \X \rangle :\:
 \langle \F_\alpha, \X \rangle
\,=\,c_\alpha;\:\X\succeq0\,\}
\,,\]
in which case the SDP solver would solve the robust optimization problem
\[\max_{\tilde{\c} \in \B_\infty(\c,\varepsilon)}\,\min_\X\,\{\, \langle \F_0, \X \rangle:\: 
\langle \F_\alpha, \X \rangle
\,=\,\tilde{c}_\alpha;\:\X\succeq0\,\} \,,
\]
where $\langle \cdot \rangle$ stands for the matrix trace and ``$\succeq 0$'' means positive semidefinite. This point of view is briefly analyzed and discussed in  Section \ref{sec:robsdp}.
\section{SDP solvers and the Moment-SOS hierarchy}
\paragraph{\textbf{Notation}}
For a fixed $j \in \N$, let us note $\R[\x]_{2 j}$ the set of polynomials of degree at most $2 j$ and $\mathcal{S}_{n,j}$ the set of real symmetric matrices of size $\binom{n+j}{n}$. For any real symmetric matrix $\M$, denote by $\Vert \M\Vert_*$ its {\em nuclear} norm and recall that if $\M\succeq0$ then $\Vert\M\Vert_*=\langle\I,\M\rangle$.
We also note $\Sigma[\x]_{j}$ for the convex cone of SOS polynomials of degree at most $2 j$.
{Let $\N^n_{2d} := \{ (\alpha_1,\dots,\alpha_n) \in \N^n : \alpha_1 + \dots + \alpha_n \leq 2 d \}$.}
In the sequel, we will use a generalization of Von Neumann's minimax theorem, namely the following Sion's minimax theorem~\cite{Sion58}:
\begin{theorem}
\label{th:sion}
Let $\B$ be a compact convex subset of a linear topological space and $\Y$ be a convex subset of a linear topological space. If $h$ is a real-valued function on $\B \times \Y$ with $h(\b,\cdot)$ lower  semi-continuous and quasi-convex on $\Y$, for all $\b \in \B$ and $h(\cdot,\y)$ upper semi-continuous and quasi-concave on $\B$, for all $\y \in \Y$, then 
\[
\max_{\b \in \B} \inf_{\y \in \Y} h(\b,\y) = \inf_{\y \in \Y} \max_{\b \in \B} h(\b,\y) 
\,. 
\]
\end{theorem}
The Moment-SOS hierarchy was introduced in~\cite{Las01sos} to solve the global
polynomial optimization problem 
\[\P: \quad f^\star\,=\,\min_\x\,\{\,f(\x): \x\in\K\,\},\]
where $f$ is a polynomial and $\K:=\{\x\in\R^n:g_l(\x)\geq0,\:l=1,\ldots,m\,\}$
is a basic closed semi-algebraic set with $(g_\ell)\subset\R[\x]$. Let us note $g_0 := 1$ and $d_\ell := \deg g_\ell$, for each $\ell=0,\ldots,m$.

A systematic numerical scheme consists of solving a hierarchy of convex relaxations:
\begin{equation}
\label{relax-primal}
\P^j:\quad \rho^j\,=\,\min_\y\,\{\,L_\y(f):\:y_0=1;\:\y\in C_j(g_1,\ldots,g_m)\,\} \,,
\end{equation}
{
where $(C_j(g_1,\ldots,g_m))_{j \in \N}$ is an appropriate nested family of convex cones, such as the one given in~\eqref{eq:Cj}.}
The dual of~\eqref{relax-primal} reads
\begin{equation}
\label{relax-dual}
\D^j: \quad \delta^j\,=\,\max_\lambda \,\{\,\lambda:\:f-\lambda \in\,C_j (g_1,\ldots,g_m)^\star \,\},
\end{equation}
where $(C_j (g_1,\ldots,g_m)^\star)_{j\in\N}\subset\R[\x]$ is a nested family of convex cones contained in
$C(\K)$, the convex cone of polynomials nonnegative on $\K$, and $L_\y:\R[\x]\to\R$
 is the Riesz Linear functional:
 \[
 f \quad \big(=\sum_\alpha f_\alpha\,x^\alpha \big) \: \mapsto L_\y(f) = \sum_\alpha f_\alpha\,y_\alpha \,.
 \]
When $C_j(g_1,\ldots,g_m)^\star$ comes from an appropriate SOS-based
(Putinar) representation of polynomials positive on $\K$, both $\P^j$  and $\D^j$ are {\em semidefinite programs} (SDP). When $\K$ is compact then (under a weak archimedean condition), $\rho^j = \delta^j \uparrow f^\star$ as $j\to\infty$, and 
generically the convergence is even finite~\cite{Nie14}, i.e., $f^\star=\rho^j$ for some $j \in \N$.  In such case, one may also extract global minimizers from an optimal solution of the corresponding semidefinite relaxation $\P^j$~\cite{Nie2013}. For more details on the Moment-SOS hierarchy, the interested reader is referred to~\cite{lasserre2009moments}.

At step $j$ in the hierarchy, one has to solve the SDP-relaxation  $\P^j$, for which efficient modern  softwares are available.
These numerical solvers all rely on interior-point methods, and are  implemented either in double precision arithmetics,
e.g.,~SeDuMi~\cite{Sturm98usingsedumi},~SDPA~\cite{Yamashita10SDPA},~Mosek~\cite{mosek}, or with arbitrary precision arithmetics, e.g.,~SDPA-GMP~\cite{Nakata10GMP}. 
When relying on such numerical frameworks, the input data considered by solvers might differ from the ones given by the user. Thus the input data, consisting of the cost vector and matrices, are subject to uncertainties. 
In~\cite{RobustElGhaoui} the authors study semidefinite programs whose input data depend on some unknown but bounded perturbation parameters. For the reader interested in robust optimization in general, we refer to~\cite{BenTal09}.

\subsection{Two examples of surprising phenomenons} 
In general, when applied for solving $\P$, the Moment-SOS hierarchy~\cite{Las01sos} is quite efficient, modulo its scalability (indeed for large size problems one has to exploit sparsity often encountered in the description of $\P$).
However, in some cases, some quite surprising phenomena have been observed
and provided additional support to the pessimistic and irrefutable conclusion that:
{\em Results returned by
double precision floating point SDP solvers cannot be trusted as they are sometimes completely wrong.} 

Let us briefly describe two such phenomena, already analyzed and commented in \cite{Waki12,Navascues13}.

{\bf Case 1:} When $\K=\R^n$ (unconstrained optimization) then the Moment-SOS hierarchy collapses to the single SDP
$\rho^d=\max_\lambda \,\{\,\lambda:\:f-\lambda \in\,\Sigma[\x]_{d}\,\}$ 
(with $2 d$ being the degree of $f$). Equivalently, one solves the semidefinite program:
\begin{equation}
\label{case1}
\D^d:\quad \rho^d\,=\,\displaystyle\max_{\X\succeq0,\lambda}\,\{\,\lambda:\:f_\alpha-\lambda\,1_{\alpha=0}\,=\,\langle \X,\B_\alpha\rangle,\quad\alpha\in\N^n_{2d}\,\}
\end{equation}
for some appropriate real symmetric matrices $(\B_\alpha)_{\alpha\in\N^n_{2d}}$; see e.g.~\cite{Las01sos}.

Only two cases can happen: if $f-f^*\in\Sigma[\x]_{d}$
then $\rho^d=f^*$ and $\rho^d<f^*$ otherwise (with possibly $\rho^d=-\infty$). 
Solving
$\rho^j=\max_\lambda \,\{\,\lambda:\:f-\lambda \in\,\Sigma[\x]_{j}\,\}$ for $j >d$ is useless as it would yield
$\rho^j=\rho^d$ because if $f-f^*$ is SOS, then it has to be in $\Sigma[\x]_{d}\subset\Sigma[\x]_j$ anyway.

The Motzkin-like polynomial 
$\x\mapsto f(\x)=x^2y^2(x^2+y^2-1)+1/27$ is nonnegative 
(with $d=3$ and $f^*=0$) and has $4$ global minimizers, but the polynomial $\x\mapsto f(\x)-f^*\,(=f)$ is {\em not} an SOS
and $\rho^3=-\infty$, which also implies $\rho^j=-\infty$ for all $j$. However, as already observed in 
\cite{Henrion2005}, by solving \eqref{case1}
%
with $j=8$ and a double precision floating point SDP solver, we obtain $\rho_{8}\approx -10^{-4}$. In addition, 
one may extract $4$ global minimizers close the global minimizers of $f$ up to four digits of precision! The same occurs with $j>8$ and the higher is $j$ the better is the result. So undoubtly the SDP solver is returning a wrong solution as $f-\rho^j$ {\em cannot} be an SOS, no matter the value of $\rho^j$.

In this case, a rationale for this behavior is that $\tilde{f} = f+\varepsilon (1+x^{16}+y^{16})$ is an SOS for small $\varepsilon>0$,
provided that $\varepsilon$ is not too small (in \cite{SOSApprox07} it is shown that every nonnegative polynomial can be approximated as closely as desired by a sequence of polynomials that are sums of squares). 
After inspection of the returned optimal solution, the 
equality constraints 
\begin{equation}
\label{equations}
f_\alpha-\lambda\,1_{\alpha=0}\,=\,\langle \X,\B_\alpha\rangle,\quad\alpha\in\N_{2 j}^n,
\end{equation}
when solving $\D^{j}$ in \eqref{case1}, 
 are not satisfied accurately and the result can be interpreted as if
the SDP solver {\em has replaced} $f$ with the perturbated criterion $\tilde{f}=f+\boldsymbol{\varepsilon}$, with $\boldsymbol{\varepsilon}(\x)=\sum_\alpha\varepsilon_\alpha\,\X^\alpha\in\R[\x]_{2d}$, so that
\[\underbrace{f_\alpha+\varepsilon_\alpha}_{\tilde{f}_\alpha}-\lambda\,1_{\alpha=0}\,=\,\langle \X,\B_\alpha\rangle,\quad\alpha\in\N_{2 j}^n,\]
 and in fact it has done so. 
A similar ``mathematical paradox'' has also been investigated in a non-commutative (NC) context~\cite{Navascues13}. NC polynomials can also be analyzed thanks to an NC variant of the Moment-SOS  hierarchy (see~\cite{burgdorf16} for a recent survey).
As in the above commutative case, it is explained in~\cite{Navascues13} how numerical inaccuracies  allow to obtain converging lower bounds for positive Weyl polynomials that do not admit SOS decompositions. 

{\bf Case 2:} Another surprising phenomenon occurred when minimizing a high-degree univariate polynomial
$f$ with a  global minimizer at $x=100$ and a local minimizer at $x=1$
with value $f(1)>f^*$ but very close to $f^*=f(100)$. The double precision floating point SDP solver
returns a single minimizer $\tilde{x}\approx 1$ with value very close to $f^*$, providing another
irrefutable proof that the double precision floating point SDP solver
has returned a wrong solution. It turns out that again the result can be interpreted as if
the SDP solver {\em has replaced} $f$ with a perturbated criterion $\tilde{f}$, as in Case~$1$.

When solving~\eqref{case1} in Case~$1$, one has voluntarily embedded $f\in\R[\x]_6$ 
into $\R[\x]_{2 j}$ (with $j > 3$)  to obtain a perturbation $\tilde{f}  \in\R[\x]_{2 j}$ whose minimizers are close enough to
those of $f$. 
Of course the precision is  in accordance with the solver parameters 
involved in controlling the semidefiniteness of the moment matrix $\X$ and the accuracy of 
the linear equations \eqref{equations}. Indeed, if one tunes 
these parameters to a much stronger threshold, then the solver returns a more accurate answer with a much higher precision.

In both contexts, we can interpret what the SDP solver does as perturbing the coefficients of the input polynomial data. One approach to get rid of numerical uncertainties consists of solving SDP problems in an exact way~\cite{Simone16}, while using symbolic computation algorithms. However, such exact algorithms only scale up to moderate size instances. For situations when one  has to rely on more efficient, yet inexact numerical algorithms, there is a need to understand the behavior of the associated numerical solvers.
In~\cite{Waki12}, the authors investigate strange behaviors of double-precision SDP solvers for semidefinite relaxations in polynomial optimization. They compute the optimal values of the SDP relaxations of a simple one-dimensional polynomial optimization problem. The sequence of SDP values practically converges to the optimal value of the initial problem while they should converge to a strict lower bound of this value. One possible remedy, used in~\cite{Waki12}, is to rely on an arbitrary-precision SDP solver, such as SDPA-GMP~\cite{Nakata10GMP} in order to make this paradoxal phenomenon disappear. Relying on such arbitrary-precision solvers comes together with a more expensive cost but paves a way towards exact certification of nonnegativity. 
In~\cite{multivsos18}, the authors present a hybrid numeric-symbolic
algorithm computing exact SOS certificates for a  polynomial lying in
the interior of the SOS cone. This algorithm uses SDP solvers to compute an
approximate SOS decomposition after additional perturbation of the coefficients of the input polynomial. The idea is to benefit from the perturbation terms added by the \textit{user} to compensate the numerical uncertainties added by the \textit{solver}. The present note focuses on analyzing specifically how the solver modifies the input and perturbates the polynomials of the initial optimization problem.

\subsection{Contribution}
We claim that there is also another possible and more optimistic conclusion if one looks at the above results with 
new ``robust optimization" glasses, not from the viewpoint of the user but rather from the view point of the solver. 
More precisely, given a polynomial optimization problem 
$f^\star=\min_\x\{ f(\x): \x\in\K\}$ and its semidefinite relaxation $\P^j$
defined in \eqref{relax-primal} (with dual $\D^j$ in \eqref{relax-dual}),
\begin{center}
{\em We interpret the above behavior as the (double precision floating point) SDP solver doing ``Robust Optimization" without being told to do so. In the case of individual trace equality perturbations $\varepsilon$, it solves the max-min problem:}
\begin{equation}
\label{robust}
\rho^j_\varepsilon \,=\,\max_{\tilde{f} \in \B_\infty^j(f,\varepsilon)} \,\{\,\inf_\y\,\{\,L_\y( \tilde{f} ):\:y_0=1;\:\y\in\,C_j(g_1,\ldots,g_m)\,\}\,\},
\end{equation}
{\em where $\B_\infty^j(f,\varepsilon):=\{\, \tilde{f} \in\R[\x]_{2 j}: \Vert \tilde{f}  -f\Vert_\infty \leq \varepsilon\,\} \,,$}
\end{center}
and we provide some numerical experiments to support this claim. Interestingly, if {\em the user} would do robust optimization, then {\em he} would solve the min-max problem:
\begin{equation}
\label{robust-minmax}
\inf_\y\,\{\,\max_{\tilde{f} \in \B_\infty^j(f,\varepsilon)} \,\{\,L_\y(\tilde{f})\,\};\:y_0=1,\: \y \in C_j(g_1,\ldots g_m)\,\} \,,
\end{equation}
which is~\eqref{robust} in which the ``$\max$" and ``$\min$" operators have been switched. 
It turns out that in this convex case, by Theorem~\ref{th:sion}, the optimal value of~\eqref{robust-minmax} is $\rho^j_\varepsilon$.

So from a {\em robustness} view point of the solver (not the user), it is quite reasonable to solve \eqref{robust} rather than
the original relaxation $\P^j$ of $\P$  with nominal polynomial $f$. However since $\rho^j_\varepsilon$ is equal to the optimal value of~\eqref{robust-minmax}, the result is the same as if the user decided to do ``robust optimization"!
In other words, solving $\P^j$ with
nominal $f$ and numerical inaccuracies is the same as solving the robust problem~\eqref{robust} or~\eqref{robust-minmax} with infinite precision.

%
%
\section{A ``noise" model}
%
Given a finite sequence of matrices $(\F_\alpha)_{\alpha \in \N^n_{2 j}} \subset \mathcal{S}_{n,j}$, a (primal) cost vector $\c = (c_\alpha)_{\alpha \in \N^n_{2 j}}$, we recall the standard form of \textit{primal} semidefinite program (SDP) solved by numerical solvers such as SDPA~\cite{Yamashita10SDPA}:
\begin{equation}
\label{primalSDP}
\begin{array}{rl}\displaystyle\min_{\y} \, &\displaystyle\sum_{\alpha \in \N^n_{2 j}} c_\alpha \, y_\alpha  \\
\mbox{s.t.}& \displaystyle\sum_{0 \neq \alpha \in \N^n_{2 j} } \F_\alpha \, y_\alpha \succeq \F_0 \,, 
\end{array}
\end{equation}
whose \textit{dual} is the following SDP optimization problem:
\begin{equation}
\label{dualSDP}
\begin{array}{rl}
\displaystyle\max_{\X} & \langle \F_0, \X\rangle\\
\mbox{s.t.}& \langle \F_\alpha, \X \rangle = c_\alpha \,, \quad \alpha \in \N^n_{2 j} \,, \quad \alpha \neq 0 \,, \\
& \X \succeq 0 \,, \quad \X \in \mathcal{S}_{n,j} \,.
\end{array}
\end{equation}
\if{
Let $\K\subset\R^{n}$ be the basic closed semi-algebraic set 
\begin{equation}
\label{setK}
\K\,:=\,\{\x \in \R^n : \: g_\ell(\x)\,\geq0,\quad \ell=1,\ldots,m\},
\end{equation}
where $(g_\ell)\subset\R[\x]$. Let us note $g_0 := 1$ and $d_\ell := \deg g_\ell$, for each $\ell=0,\ldots,m$.
}\fi
We are interested in the numerical analysis of the moment-SOS hierarchy~\cite{Las01sos} to solve
\[
\P:\quad \min_{\x \in \K} \,f(\x) \,,
\]
where $f\in\R[\x]_{2 j}$. 
Given $\alpha,\beta \in \N^n$, let $1_{\alpha = \beta}$ stands for the function which returns 1 if $\alpha = \beta$ and 0 otherwise.
At step $d$ of the hierarchy, one solves the SDP primal program \eqref{relax-primal}.
For the standard choice of the convex cone $C_j(g_1,\ldots,g_m)$, given by 
\begin{align}
\label{eq:Cj}
C_j(g_1,\ldots,g_m) = \{\y : \M_{j-d_\ell}(g_\ell\,\y)\,\succeq0,\quad \ell=0,\ldots,m \} \,,
\end{align}
it reads
\begin{equation}
\label{primalj}
\displaystyle\rho^j=\inf_{\y} \, \,\{\,L_\y(f) :\:y_0=1;\quad \M_{j-d_\ell}(g_\ell\,\y)\,\succeq0,\quad \ell=0,\ldots,m\,\} \,,
\end{equation}
whose dual is the SDP:
\begin{equation}
\label{dualj}
\begin{array}{rl}
\displaystyle\delta^j=\sup_{\X_\ell, \lambda}\,\{\,\lambda:&\:f_\alpha-\lambda\,1_{\alpha=0}=\displaystyle\sum_{\ell=0}^m \langle \C^\ell_\alpha, \X_\ell \rangle \,, \quad \alpha \in \N^n_{2j} \,, \\
&  \X_\ell \succeq 0 \,, \quad \X_\ell \in \S_{n,j-d_\ell} \,, \quad \ell=0,\ldots,m\,\}
\end{array},
\end{equation}
where we have written $\M_{j-d_\ell}(g_\ell\,\y) = \sum_{\alpha \in \N^n_{2j}} \C^\ell_\alpha \, y_\alpha$; the matrix $\C^\ell_\alpha$ has rows and columns indexed by $\N^n_{j- d_\ell}$ with $(\beta,\gamma)$ entry equal to $\sum_{\beta+\gamma+\delta = \alpha} \, g_{\ell,\delta}$. In particular for $m=0$, one has $g_0 = 1$ and the matrix $\B_\alpha :=  \C^0_\alpha$ has $(\beta,\gamma)$ entry equal to $1_{\beta + \gamma = \alpha}$.

For every $j\in\N$, let 
\[\Q_j(g)\,=\, \left\lbrace \,\sum_{\ell=0}^m \sigma_\ell\,g_\ell:\:{\rm deg}(\sigma_\ell\,g_\ell)\,\leq 2 j \,, \sigma_\ell\in\Sigma[\x]\, \right\rbrace \]
be the ``truncated" quadratic module associated with the $g_\ell$'s.

Then the dual SDP~\eqref{dualj} can be rewritten as
\begin{equation}
\label{dualjsos}
\begin{array}{rl}
\delta^j=\displaystyle\sup_{\lambda} \: \{\lambda : f- \lambda \in \Q_j(g) \} = \displaystyle\sup_{\lambda,\sigma_\ell} \,\{\,\lambda:&
 f- \lambda = \displaystyle\sum_{\ell=0}^m \sigma_\ell\,g_\ell \,, \\
& \deg (\sigma_\ell\,g_\ell) \, \leq 2 j \,, \quad \sigma_\ell\in\Sigma[\x] \,\}.
\end{array}
\end{equation}
{
Strong duality of Lasserre's hierarchy is guaranteed when the following condition (slightly stronger than compactness of $\K$) holds:
\begin{assumption}
\label{hyp:archimedean}
There exists $N \in \N$ such that one of the polynomials {describing the set $\K$ reads $g^\K(\x) := N - \| \x \|_2^2$.}
\end{assumption}
Then it follows from~\cite{Josz2016} that this ball constraint implies strong duality between~\eqref{primalj} and~\eqref{dualjsos}.
Note that if the set $\K$ is bounded, then one can add the redundant constraint $N - \| \x \|_2^2 \geq 0$, without modifying $\K$.
In the sequel, we suppose that Assumption~\ref{hyp:archimedean} holds.
}

In floating point computation, the numerical SDP solver treats 
all (ideally) equality constraints as the following inequality constraints
\begin{equation}
\label{equality}
\begin{array}{rl}
 \displaystyle\sum_{\ell=0}^m \langle \C^\ell_\alpha, \X_\ell \rangle + \lambda 1_{\alpha=0} - f_\alpha = 0 \,, \quad \alpha \in \N^n_{2 j} \,,
\end{array}
\end{equation}
of ~\eqref{dualj} with the following inequality constraints
\begin{equation}
\label{eq-approx}
\begin{array}{rl}
\biggl|  
\displaystyle\sum_{\ell=0}^m \langle \C^\ell_\alpha, \X_\ell \rangle + \lambda 1_{\alpha=0} - f_\alpha
\biggr| \leq \varepsilon \,, \quad \alpha \in \N^n_{2 j} \,,
\end{array}
\end{equation}
for some a priori fixed tolerance $\varepsilon > 0$ (for instance $\varepsilon=10^{-8}$). 
Similarly, we assume that for each $\ell=0,\dots,m$, the SDP constraint
$\X_\ell \succeq 0$ of ~\eqref{dualj}  is relaxed to $\X_\ell \succeq -\eta \, \I$ for some prescribed {\em individual semidefiniteness tolerance} $\eta >0$. 
{This latter relaxation of $\succeq0$ to $\succeq-\eta\I$ is used here as an idealized situation for modeling purpose;  in practice it seems to be more complicated, as explained later on at the beginning of Section~\ref{sec:benchs}.}

That is, all iterates $(\X_{\ell,k})_{k\in\N}$ of the implemented minimization algorithm satisfy  \eqref{eq-approx} and $\X_{\ell,k}\succeq-\eta\I$ instead of the idealized \eqref{equality} and $\X_{\ell,k}\succeq0$.

Therefore we interpret the SDP solver behavior by considering
 the following ``noise" model which is the $(\varepsilon,\eta)$-perturbed version of SDP~\eqref{dualj}:
\begin{equation}
\label{dualjepseta}
\begin{array}{rl}
\displaystyle\sup_{\X_\ell,\lambda}\,\{\lambda:\:&
 -\varepsilon \leq \displaystyle
\sum_{\ell=0}^m \langle \C^\ell_\alpha, \X_\ell \rangle + \lambda 1_{\alpha=0} - f_\alpha \leq \varepsilon
\,, \quad \alpha \in \N^n_{2j} \,, \\
& \X_\ell \succeq - \eta \, \I \,, \quad \X_\ell \in \S_{n,j-d_\ell} \,, \quad \ell=0,\ldots,m\,\}, 
\end{array}
\end{equation}
now assuming exact computations.
\begin{proposition}
\label{th:primaldualepseta}
The dual of Problem~\eqref{dualjepseta} is the convex optimization problem
\begin{equation}
\label{primaljepseta}
\begin{array}{rl}
\displaystyle\inf_{\y} & \{\,L_\y(f) +  \eta \displaystyle\sum_{\ell=0}^m \Vert\M_{j - d_\ell} (g_\ell \, \y)\Vert_* +  \varepsilon \Vert \y \Vert_1 :\\
\mbox{s.t.} & y_0=1;\quad \M_{j-d_\ell}(g_\ell\,\y)\,\succeq0,\quad \ell=0,\ldots,m\,\}\\
\end{array}
\end{equation}
which is an SDP.
\end{proposition}
\begin{proof}
Let $y_\alpha^\pm$ be the nonnegative dual variables associated with the constraints
\[\pm\,\left(\sum_{\ell=0}^m \langle \C^\ell_\alpha, \X_\ell \rangle + \lambda 1_{\alpha=0} - f_\alpha\,\right) \leq \varepsilon,\quad\alpha\in\N^n_{2 j},\]
and let $\S_\ell\succeq0$ be the dual matrix variable associated with the SDP constraint
$\X_\ell\succeq -\eta\,\I$, $\ell=0,\ldots,m$. Then the dual of~\eqref{dualjepseta} is a semidefinite program which reads:
\begin{equation}
\label{sdp-primal-player2}
\begin{array}{rl}\displaystyle
\inf_{\S_\ell\succeq0,y^\pm_\alpha\geq0}&\big\{\,\displaystyle\sum_\alpha (f_\alpha\,(y^+_\alpha-y^-_\alpha)+\varepsilon\,(y^+_\alpha+y^-_\alpha))+\eta\,\displaystyle\sum_\ell \langle \I,\S_\ell\rangle:\\
& \S_\ell-\sum_\alpha\C_\alpha^\ell\,(y^+_\alpha-y^-_\alpha)\,=\,0,\quad \ell=0,\ldots,m \,, \\
&y^+_0-y^-_0\,=1  \big\} \,.
\end{array}
\end{equation}
In view of the nonnegative terms $\varepsilon\,\sum_\alpha(y^+_\alpha+y^-_\alpha)$
in the criterion, at an optimal solution we necessarily have $y^+_\alpha\,y^-_\alpha = 0$, for all $\alpha$. Therefore letting 
$y_\alpha:=y^+_\alpha-y^-_\alpha$, on obtains $y^+_\alpha+y^-_\alpha=\vert y_\alpha\vert$ for all $\alpha$, and $\sum_\alpha (y^+_\alpha+y^-_\alpha)=\Vert \y\Vert_1$. Similarly
as $\S_\ell\succeq0$, $\langle \I,\S_\ell\rangle=\Vert \S_\ell \Vert_1$, $\ell=0,\ldots,m$.
This yields the formulation \eqref{primaljepseta}.
\end{proof}
\begin{remark}
{
Notice that the criterion of \eqref{primaljepseta} consists of the original criterion
$L_\y(f)$ perturbated with a sparsity-inducing norm $\varepsilon\,\Vert\y\Vert_1$ for the variable $\y$ and a low-rank-inducing norm $\eta\,\sum_\ell\Vert\M_{j-d_\ell}(g_\ell\,\y)\Vert_*$ for the localizing matrices.
Considering this low-rank-inducing term can be seen as the convexification of a more realistic penalization with a logarithmic barrier function used in interior-point methods for SDP, namely $- \eta \log \det \big( \sum_{\ell=0}^m \M_{j - d_\ell} (g_\ell \, \y) \big)$.
One could also consider to replace each SDP constraint $\M_{j - d_\ell} (g_\ell \, \y) \succeq 0$ with $\M_{j - d_\ell} (g_\ell \, \y) \succeq \varepsilon_3 \I$, in the primal moment problem~\eqref{primalj}. 
This corresponds to add $ -\varepsilon_3 \Vert \X \Vert_*$ in the related perturbation of the dual SOS problem~\eqref{dualj}. 
One can in turn interpret this term as a convexification of the more standard logarithmic barrier penalization term $\log \det \X$.
Even though interior-point algorithms could practically perform such logarithmic barrier penalizations, we do not have a simple interpretation for the related noise model. 
}
\end{remark}

We now distinguish among two particular cases.

\subsection{Priority to trace equalities}

With $\varepsilon = 0$ and individual semidefiniteness-tolerance 
$\eta$,  Problem~\eqref{primaljepseta} becomes
\begin{equation}
\label{primaljeta}
\begin{array}{rl}
\rho^j_\eta= \ \displaystyle\inf_{\y} & \{\,L_\y(f) +  \eta \displaystyle\sum_{\ell=0}^m 
\Vert\M_{j - d_\ell} (g_\ell \, \y)\Vert_*  \\
\mbox{s.t.} & y_0=1;\quad \M_{j-d_\ell}(g_\ell\,\y)\,\succeq0,\quad \ell=0,\ldots,m\,\}.
\end{array}
\end{equation}
\if{
Let us define $\Q_j^-(g,\eta)$ as follows: 
\[\Q_j^-(g,\eta) :=\, \left\lbrace \,\sum_{\ell=0}^m (\sigma_\ell -\eta \sum_{\beta \in \N^n_{j - d_\ell}} \x^{2 \beta}) \,g_\ell:\:{\rm deg}(\sigma_\ell\,g_\ell)\,\leq 2 j \,, \sigma_\ell\in\Sigma[\x] \, \right\rbrace \,. \]
Then the dual of SDP~\eqref{primaljeta} reads
\begin{equation}
\label{dualjeta}
\sup_{\lambda} \: \{ \lambda \: : f-\lambda \, \in \, 
\Q_j^-(g,\eta) \} \,.
\end{equation}
SDP~\eqref{primaljeta} aims at finding a low rank solution $\y$.
Fix $j \in \N$ and consider the following robust polynomial optimization problem 
\begin{align}
\label{pop_ineq}
\P_{\eta}^{\max} : \quad \min_{\x \in \K} \: \biggl[ f(\x)+ \eta \,\sum_{\ell=0}^m\sum_{\beta\in\N^n_{j-d_\ell}}\x^{2\beta}\,g_\ell(\x) \biggr] 
\,.
\end{align}
}\fi
Given $\eta > 0$, $j \in \N$, let us define: 
\begin{align}
\label{eq:BallPos}
\B_\infty^j (f, \K, \eta) & := \{\, f + \theta \sum_{\ell = 0}^m g_\ell(\x) \sum_{\beta \in \N^n_{j - d_\ell}} \x^{2 \beta}  :  |\theta| \leq \eta  \,\} \,, \\
\B_\infty (f, \K, \eta) & := \bigcup_{j \in \N}  \B_\infty^j (f, \K, \eta)  \nonumber
\,. 
\end{align}
Recall that SDP~\eqref{primaljeta} is the dual of SDP~\eqref{dualjepseta} with $\varepsilon=0$, that is,

\begin{equation}
\label{dualjeta}
\begin{array}{rl}
\displaystyle\sup_{\X_\ell,\lambda}\,\{\lambda:\:&
  \displaystyle
 f_\alpha - \lambda 1_{\alpha=0} = \sum_{\ell=0}^m \langle \C^\ell_\alpha, \X_\ell \rangle 
\,, \quad \alpha \in \N^n_{2j} \,, \\
& \X_\ell \succeq - \eta \, \I \,, \quad \X_\ell \in \S_{n,j-d_\ell} \,, \quad \ell=0,\ldots,m\,\}, 
\end{array}
\end{equation}

\if{
\begin{equation}
\label{dualjeta}
\begin{array}{rl}
\rho_j=\displaystyle\sup_{\tilde{f} ,\lambda} & \{\,\lambda :\: f - \lambda \in \Q_j(g);\quad \vert f_\alpha - \tilde{f}_\alpha\vert \leq \varepsilon \,, \quad \alpha \in \N^n_{2 j}  \,, \\
& \lambda \in \R \,, \quad \tilde{f} \in \R[\x]_{2 j} \,\}.
\end{array}
\end{equation}
}\fi


Fix $j \in \N$ and consider the following robust polynomial optimization problem 

\begin{equation}
\label{maxmin_eq}
\P^{\max}_{\eta}:\quad \max_{\tilde{f} \: \in \B_\infty (f, \K, \eta) }\,\{\,\min_{\x\in\K }\,\{ \tilde{f}  (\x)\} \, \} \,.
\end{equation}
If in~\eqref{maxmin_eq}, we restrict ourselves to $\B^j_\infty (f, \K, \eta)$ and we replace the inner
minimization by its step-$j$ relaxation, we obtain
\begin{eqnarray}
\label{maxmin_eq-relax}
\P^{\max,j}_{\eta}:\quad \max_{\tilde{f} \: \in \B_\infty^j (f, \K, \eta) }\,\Big\{\, \inf_{\y}\, \{\,L_\y(\tilde{f}) : y_0=1; \ \M_{j-d_\ell}(g_\ell\,\y)\,\succeq0, \ \ell=0,\ldots,m\,\} \Big\} \,. \nonumber
\end{eqnarray}
Observe that Problem~$\P^{\max,j}_{\eta}$ is a strenghtening of Problem~$\P^{\max}_{\eta}$, that is, the optimal value of the former is smaller than the optimal value of the latter.
%
\begin{proposition}
\label{th:pop_ineq}
{Under Assumption~\ref{hyp:archimedean}, there is no duality gap between primal SDP~\eqref{primaljeta} and  dual SDP~\eqref{dualjeta}.}
In addition, Problem~$\P^{\max,j}_{\eta}$ is equivalent to   SDP~\eqref{primaljeta}. Therefore, solving primal SDP~\eqref{primaljeta} (resp.~dual SDP~\eqref{dualjeta}) can be interpreted as solving {\em exactly}, i.e., with no 
semidefiniteness-tolerance, the step-$j$ strenghtening~$\P^{\max,j}_{\eta}$ associated with Problem~$\P^{\max}_{\eta}$.
\end{proposition}
\begin{proof}
Remind that for every $\ell=0,\dots,m$, one has $\M_{j-d_\ell}(g_\ell\,\y)\,\succeq0$ and 
\[
\Vert\M_{j-d_\ell}(g_\ell\,\y)\Vert_*\,=\,\trace{(\M_{j - d_\ell} (g_\ell \, \y) )} = L_\y \biggl(\sum_{\beta\in\N^n_{j-d_\ell}}\x^{2\beta}\,g_\ell(\x) \biggr) \,.
\]
For $\tilde{f} = f + \eta \sum_{\beta\in\N^n_{j-d_\ell}}\x^{2\beta}\,g_\ell(\x)$, one has $L_\y(\tilde{f}) = L_\y(f) + \eta \sum_{\ell=0}^m \Vert\M_{j-d_\ell}(g_\ell\,\y)\Vert_*$.
{
Thus, the primal SDP~\eqref{primaljeta} (resp.~dual SDP~\eqref{dualjeta}) boils down to solving the primal SDP~\eqref{primalj} (resp.~\eqref{dualj}) after replacing $f$ by $\tilde{f}$. By Assumption~\ref{hyp:archimedean}, there is no duality gap between~\eqref{primalj} and~\eqref{dualj}, thus there is also no duality gap between~\eqref{primaljeta} and~\eqref{dualjeta}.}

In addition, since $\tilde{f}$ is feasible for Problem~$\P^{\max,j}_{\eta}$, the optimal value of Problem~$\P^{\max,j}_{\eta}$ is greater than the value of SDP~\eqref{primaljeta}.
By Theorem~\ref{th:sion}, Problem~$\P^{\max,j}_{\eta}$ is equivalent to 
\begin{equation}
\label{maxmin_eq-bis}
\inf_{\y}\, \max_{\tilde{f} \: \in \B_\infty^j (f, \K, \eta)  }\, \{\,L_\y(\tilde{f}) : y_0=1; \ \M_{j-d_\ell}(g_\ell\,\y)\,\succeq0, \ \ell=0,\ldots,m\,\}  \,. 
\end{equation}
For all $\tilde{f} \: \in \B_\infty^j (f, \K, \eta)$, $L_\y(\tilde{f}) \leq L_\y(f) + \eta \sum_{\ell=0}^m \Vert\M_{j-d_\ell}(g_\ell\,\y)\Vert_*$, which proves that the optimal value of~\eqref{maxmin_eq-bis} is less than the value of SDP~\eqref{primaljeta}.

This yields the equivalence between Problem~$\P^{\max,j}_{\eta}$ and SDP~\eqref{primaljeta}.
%
%
\end{proof}
%
In the unconstrained case, i.e.~when $m = 0$,  solving $\P_\eta^{\max,j}$ boils down to minimize the  perturbed polynomial $f_{\eta,j}(\x) := f(\x)+ \eta \, \sum_{\vert\beta\vert \leq j}\x^{2\beta}$, that is the sum of $f$ and all monomial squares of degree up to $2 j$ with coefficient magnitude $\eta$.
As a direct consequence from~\cite{SOSApprox07}, the next result shows that for given nonnegative polynomial $f$ and perturbation $\eta > 0$, the polynomial $f_{\eta,j}$ is SOS for large enough $j$.
\begin{corollary}
\label{th:perturbSOS}
Let assume that $f \in \R[\x]$ is nonnegative over $\R^n$ and let us fix $\eta > 0$. Then $f_{\eta,j}\in \Sigma[\x]$, for large enough $j$.
\end{corollary}
\begin{proof}
For fixed nonnegative $f \in \R[\x]$ and $\eta > 0$, it follows from~\cite[Theorem~4.2~(ii)]{SOSApprox07} that there exists $j_\eta$ (depending on $f$ and $\eta$) such that the polynomial
\[
 f + \eta \, \sum_{k=0}^{j} \sum_{i=1}^n \frac{x_i^{2 k}}{k!} \,,
\]
is SOS for any $d \geq d_\eta$. Let us select $j := j_\eta$. Notice that 
\[
f_{\eta,j} = f + \eta \, \sum_{|\beta| \leq j}\x^{2\beta} = f + \eta \, \sum_{k=0}^j \sum_{i=1}^n \frac{x_i^{2 k}}{k!} + \eta \, \sum_{k=0}^j \sum_{i=1}^n \biggl( 1 - \frac{1}{k!} \biggr) \, x_i^{2 k} + \eta \, q_j \,,
\]
where $q_j$ is a sum of monomial squares. Since $(1 - \frac{1}{k!}) \geq 0$, the second sum of the right hand side is SOS, yielding the desired claim.
\end{proof}

\subsection{Priority to semidefiniteness inequalities} 
Problem~\eqref{primaljepseta} with $\eta = 0$ and individual trace equality perturbation $\varepsilon$ becomes
\begin{equation}
\label{primaljeps}
\begin{array}{rl}
\rho^j_\varepsilon = \displaystyle\inf_{\y} & \{\,L_\y(f) +  \varepsilon \,\Vert \y \Vert_1 \,:\\
\mbox{s.t.} & y_0=1;\quad \M_{j-d_\ell}(g_\ell\,\y)\,\succeq 0,\quad \ell=0,\ldots,m\,\}.\\
\end{array}
\end{equation}
Given $\varepsilon>0$, $j\in\N$, let us define 
\begin{align}
\label{eq:BallPol}
\B_\infty^j (f, \varepsilon) := \{\, \tilde{f} \in \R[\x]_{2 j}:\Vert f- \tilde{f} \Vert_\infty \leq \varepsilon\,\} \,, \quad 
\B_\infty (f, \epsilon):= \bigcup_{j \in \N}  \B_\infty^j (f, \epsilon)
\,. 
\end{align}
Recall that \eqref{primaljeps} is the dual of \eqref{dualjepseta} with $\eta=0$, that is,
\begin{equation}
\label{dualjeps}
\begin{array}{rl}
\displaystyle\sup_{\tilde{f} ,\lambda} & \{\,\lambda :\: \tilde{f} - \lambda \in \Q_j(g);\quad \vert f_\alpha - \tilde{f}_\alpha\vert \leq \varepsilon \,, \quad \alpha \in \N^n_{2 j}  \,, \\
& \lambda \in \R \,, \quad \tilde{f} \in \R[\x]_{2 j} \,\}.
\end{array}
\end{equation}

Fix $j\in\N$ and consider the following robust polynomial optimization problem:
\begin{equation}
\label{maxmin_ineq}
\P^{\max}_{\varepsilon}:\quad \max_{\tilde{f} \: \in \B_\infty (f, \epsilon) }\,\{\,\min_{\x\in\K }\,\{ \tilde{f}  (\x)\}\,\}.
\end{equation}
If in~\eqref{maxmin_ineq}, we restrict ourselves to $\B_\infty^j (f, \epsilon)$ in the outer maximization problem and we replace the inner
minimization by its step-$j$ relaxation, we obtain
\begin{eqnarray}
\nonumber
\P^{\max,j}_{\varepsilon}&:& 
\max_{\tilde{f} \: \in \B_\infty^j (f, \epsilon) }\,\{\,\sup_{\lambda}\,
\{\,\lambda :\: \tilde{f}-\lambda\,\in\,\mathcal{Q}_j(g)\}\,\}\\
\label{maxmin_ineq-strength}
&=&\max_{\tilde{f} \: \in \B_\infty^j (f, \epsilon) }\,
\{\,\inf_{\y} \,\{\,L_\y(\tilde{f}):\:y_0=1;\:\M_j(g_\ell\,\y)\succeq0,\:\ell=0,\ldots,m\}\,\}
\end{eqnarray}
{Here, we rely again on Assumption~\ref{hyp:archimedean} to ensure strong duality and obtain~\eqref{maxmin_ineq-strength}.}
Problem~$\P^{\max,j}_{\varepsilon}$ is a strengthening of $\P^{\max}_\varepsilon$ and whose dual
is exactly \eqref{primaljeps}, that is:
\begin{proposition}
\label{th:primaljeps}
Under Assumption~\ref{hyp:archimedean}, solving~\eqref{primaljeps} (equivalently~\eqref{dualjeps}) can be interprated as solving {\em exactly}, i.e.~with no trace-equality tolerance, the step-$j$ 
reinforcement~$\P^{\max,j}_\varepsilon$ associated with~$\P^{\max}_\varepsilon$. 
\end{proposition}
%

%

%

\subsection{A two-player game interpretation}
\label{sec:two-player}
If we now assume that one can perform computations exactly,
we can interpret the whole process in~$\P^{\max,j}_\eta$ (resp.~$\P^{\max,j}_\varepsilon$) 
 as a two-player zero-sum game in which:
\begin{itemize}
\item  Player 1 (the solver) 
chooses a polynomial $\tilde{f} \in \B_\infty^j (f,\K,\eta)$ (resp.~$\tilde{f} \in \B_\infty^j (f,\varepsilon)$).
\item Player 2 (the optimizer) then selects a minimizer $\y^\star(\tilde{f})$ in the inner minimization of~\eqref{maxmin_ineq-strength}, e.g., with an exact interior point method.
\end{itemize}
As a result, Player $1$ (the leader) 
obtains an optimal polynomial $\tilde{f}^\star \in\B^j_\infty(f,\K,\eta)$ (resp.~$\tilde{f}^\star \in\B^j_\infty(f,\varepsilon)$)
and Player $2$  (the follower) obtains an associated minimizer $\y^\star(\tilde{f}^\star)$. \\
The polynomial $\tilde{f}^\star$ is the {\em worst} polynomial in $\B^j_\infty(f,\K,\eta)$ (resp.~$\B^j_\infty(f,\varepsilon)$) for the step-$j$ semidefinite relaxation associated with the optimization problem $\min_\x\{\, \tilde{f}(\x): \x\in\K\}$. This $\max-\min$ problem
is then equivalent to the single $\min$-problem~\eqref{primaljeta} (resp.~\eqref{primaljeps}) which is a convex relaxation and whose convex criterion is not linear as it contains the sum of $\ell_\infty$-norm  terms $\sum_{\ell=0}^m 
\Vert\M_{j - d_\ell} (g_\ell \, \y)\Vert_*$ (resp.~the $\ell_1$-norm  term $\Vert\y\Vert_1$).

Notice that in this scenario the optimizer (Player $2$) is {\it not} active; initially he wanted to solve the convex relaxation associated with $f$. It is Player $1$ (the adversary uncertainty in the solver) who in fact 
{\em gives} the exact algorithm his own choice of the function $\tilde{f} \in\B^j_\infty(f,\K,\eta)$ (resp.~$\tilde{f} \in\B^j_\infty(f,\varepsilon)$).
But in fact, as we are in the convex case,  Theorem~\ref{th:sion} implies that this $\max-\min$ game is also equivalent to
the $\min-\max$ game. Indeed, $\P^{\max,j}_{\eta}$ is equivalent to 
\[
\inf_{\y} \,\max_{\tilde{f} \: \in \B_\infty^j (f, \K, \eta) }\,\{
\,L_\y(\tilde{f}):\:y_0=1;\:\M_j(g_\ell\,\y)\succeq0,\:\ell=0,\ldots,m\,\}
\,, \]
and $\P^{\max,j}_{\varepsilon}$ is equivalent to 
\[
\inf_{\y} \,\max_{\tilde{f} \: \in \B_\infty^j (f, \varepsilon) }\,\{
\,L_\y(\tilde{f}):\:y_0=1;\:\M_j(g_\ell\,\y)\succeq0,\:\ell=0,\ldots,m\,\}
\,, \]

\if{
\begin{proposition}
\label{convexity}
$\P^{\max,j}_{\varepsilon}$ is equivalent to
\begin{equation}
\label{equivPmaxjeps}
\inf_{\y} \,\sup_{\tilde{f} \: \in \B_\infty^j (f, \epsilon) }\,\{
\,L_\y(\tilde{f}):\:y_0=1;\:\M_j(g_\ell\,\y)\succeq0,\:\ell=0,\ldots,m\,\} \,.
\end{equation}
\end{proposition}
\begin{proof}
Observe that~\eqref{equivPmaxjeps} is equivalent to ~\eqref{primaljeps}, and so the result follows from Proposition \ref{th:primaljeps}.
\end{proof}
}\fi
So now in this scenario (which assumes exact computations):
\begin{itemize}
\item  Player 1 (the robust optimizer) chooses a feasible moment sequence $\y$ with $\y_0=1$ and $\M_{j-d_\ell}(g_\ell\,\y)\succeq0$, $\ell=0,\ldots,m$.
\item When priority is given to trace equalities, Player 2 (the solver) then selects $\tilde{f} (\y)=\arg\max\{ L_\y(\tilde{f}): \tilde{f} \in\B^j_\infty(f,\K,\eta)\}$ 
to obtain the value $L_\y(f)+\eta \sum_{\ell=0}^m 
\Vert\M_{j - d_\ell} (g_\ell \, \y)\Vert_*$.
\\
When priority is given to semidefinitess inequalities, Player 2 selects $\tilde{f} (\y)=\arg\max\{ L_\y(\tilde{f}): \tilde{f} \in\B^j_\infty(f,\varepsilon)\}$ to obtain the value  $L_\y(f)+\varepsilon\Vert\y\Vert_1$, that is 
$ \tilde{f} (\y)_\alpha= f_\alpha + {\rm sign} (y_\alpha) \: \varepsilon$, $\alpha\in\N^n_{2j}$.
\end{itemize}
Here  the optimizer (now Player $1$) is ``active" as {\em he} decides to compute a ``robust" optimal relaxation $\y$ assuming uncertainty in the function $f$ 
in the criterion $L_\y(f)$.

Since both scenarii are equivalent it is fair to say that the SDP solver 
is indeed solving the robust convex relaxation that the optimizer whould have 
given to a solver with exact arithmetic (if he had wanted to solve robust relaxations)
\subsection*{Relating to robust optimization}
Suppose that there is no computation errror but we want to solve a robust version of the optimization problem $\min \{f(\x):\x\in\K\}$
because there is some uncertainty in the coefficients of the {\em nominal} polynomial $f\in\R[\x]_{d}$. So assume that $f\in\R[\x]_{d}$  can be considered as potentially of degree at most $2 j$ (after perturbation).  

When priority is given to trace equalities, the robust optimization problem reads:
\begin{equation}
\label{robust-infty-eq}
\P^{\min, j}_{\eta}:\quad \min_{\x\in\K}\,\{\,\max_{\tilde{f}  \in \B_\infty^j(f, \K,\eta)}\,\{ \tilde{f} (\x)\}\,\}.
\end{equation}
Straightforward calculation reduces \eqref{robust-infty-eq} to:
\begin{equation}
\label{robust-infty-new-eq}
\P^{\min,j}_{\eta}:\quad \min_{\x\in\K} \, \bigl[ \,f(\x)+ \eta \displaystyle\sum_{\beta\in\N^n_{j-d_\ell}}\x^{2\beta}\,g_\ell(\x) \, \bigr] \,.
\end{equation}
which is a polynomial optimization problem. 

%
%

%
%
\begin{theorem}
\label{th:exact_eq}
Suppose that Assumption~\ref{hyp:archimedean} holds.
Assume that after solving SDP~\eqref{primaljeta}, one obtains $\y^\star$ such that $\M_j(\y^\star)$ is a rank-one matrix. Then, the optimal value of $\P^{\min,j}_{\eta}$ is equal to $\rho_\eta^j$ and $\P^{\min,j}_{\eta}$ is equivalent to $\P^{\max,j}_{\eta}$.
\end{theorem}
\begin{proof}
Since  $\M_j(\y^\star)$ is a rank-one matrix, the sequence $\y^\star$ comes from a Dirac measure supported on $\x^\star \in \K$. Then one has 
\[
L_{\y^\star}(f)+ \eta \, \displaystyle\sum_{\ell=0}^m 
\Vert\M_{j - d_\ell} (g_\ell \, \y^\star)\Vert_* =
f(\x^\star) + \eta \displaystyle\sum_{\beta\in\N^n_{j-d_\ell}}\x^{\star 2\beta}\,g_\ell(\x^\star) \,,
\]
Let $\mathcal{P}(\K)$ be the space of probability measures supported on $\K$. Then, one has
\begin{align*}
f(\x^\star) + \eta \displaystyle\sum_{\ell=0}^m \sum_{\beta\in\N^n_{j-d_\ell}}\x^{\star 2\beta}\,g_\ell(\x^\star)
& \geq 
\min_{\x\in\K} \, \bigl[ \,f(\x) + \eta \displaystyle\sum_{\ell=0}^m\sum_{\beta\in\N^n_{j-d_\ell}}\x^{2\beta}\,g_\ell(\x) \, \bigr] \\
& = 
\inf_{\mu \in \mathcal{P}(\K)} 
\Bigl[ \int f d \mu + \eta \displaystyle\sum_{\ell=0}^m\sum_{\beta\in\N^n_{j-d_\ell}}\int  \x^{2\beta}\,g_\ell(\x) d \mu \Bigr]
\\
& \geq \rho^j_\eta  = L_{\y^\star}(f)+ \eta \, \displaystyle\sum_{\ell=0}^m 
\Vert\M_{j - d_\ell} (g_\ell \, \y^\star)\Vert_* \,.
\end{align*}
This implies that $\x^\star$ is the unique optimal solution of~$\P^{\min,j}_{\eta}$ and that the  optimal value of~$\P^{\min,j}_{\eta}$ is equal to $\rho^j_\eta$. Eventually, Proposition~\ref{th:pop_ineq} yields the desired equivalence.
%
%
\end{proof}

When priority is given to semidefiniteness inequalities, the robust optimization problem reads:
\begin{equation}
\label{robust-infty}
\P^{\min,j}_{\varepsilon}:\quad \min_{\x\in\K}\,\{\,\max_{\tilde{f}  \in \B_\infty^j(f, \varepsilon)}\,\{ \tilde{f} (\x)\}\,\}.
\end{equation}
It is easy to see that~\eqref{robust-infty} reduces to
\begin{equation}
\label{robust-infty-new}
\P^{\min,j}_{\varepsilon}:\quad \min_{\x\in\K} \, \bigl[ \,f(\x)+ \varepsilon \displaystyle\sum_{\alpha\in\N^n_{2j}}\vert\x^\alpha\vert \, \bigr] \,.
\end{equation}
which is {\em not} a polynomial optimization problem
(but is still a semi-algebraic optimization problem). 
As for Theorem~\ref{th:exact_eq}, one proves the following result:
\begin{theorem}
\label{th:exact_ineq}
Suppose that Assumption~\ref{hyp:archimedean} holds. 
Assume that after solving SDP~\eqref{primaljeps}, one obtains $\y^\star$ such that $\M_j(\y^\star)$ is a rank-one matrix. Then, the optimal value of $\P^{\min,j}_{\varepsilon}$ is equal to $\rho_\varepsilon^j$ and $\P^{\min,j}_{\varepsilon}$ is equivalent to $\P^{\max,j}_{\varepsilon}$.
\end{theorem}
%


Notice an important conceptual difference between the two approaches.
In the latter one, i.e.~when considering $\P^{\min}_{\eta}$ (resp.~$\P^{\min}_{\varepsilon}$), the user is active. Indeed the user decides to choose some optimal $\hat{f}\in \B_\infty^j(f,\K,\eta)$ (resp.~$\B_\infty^j(f,\varepsilon)$). 
In the former one, i.e., when considering $\P^{\max}_{\eta}$ (resp.~$\P^{\max}_{\varepsilon}$), the user is passive, as indeed he imposes $f$ but the solver  decides to choose some optimal $f^\star \in \B_\infty^j(f,\K,\eta)$ (resp.~$\B_\infty^j(f,\varepsilon)$).\\
If after solving SDP~\eqref{primaljeta} (resp.~SDP~\eqref{primaljeps}), one obtains $\y^\star$ where $\M_j(\y^\star)$ is rank-one (which is to be expected), one obtains the same solution: 
in other words, we can interpret what the solver does as performing robust polynomial optimization.

In the sequel, we show how this interpretation relates with a more general robust SDP framework, when priority is given to semidefinitess inequalities.
{
\subsection{Link with robust semidefinite programming}
\label{sec:robsdp}

Let $\c = (c_j) \in \R^n$, $\F_j$ be real symmetric $t\times t$ 
matrix, $j=0,1,\ldots,n$, and let $\F (\y) := \sum_{j=1}^n\F_j \, y_j- \F_0$. 
Consider the canonical semidefinite program (SDP):
\begin{equation}
\label{pb-sdp}
\P:\quad \displaystyle\inf_\y\,\{\,\c^T\y:\:\F(\y)\,\succeq\,0\,\}
\end{equation}
with dual 
\begin{equation}
\label{pb-sdp-dual}
\P^*:\quad \displaystyle\sup_{\X\succeq0}\,\{\langle\F_0,\X\rangle:\:\langle\F_j,\X\rangle\,=\,c_j,\quad j=1,\ldots,n\,\}.
\end{equation}

Given $\varepsilon> 0$ fixed, let
$\B_\infty(\c, \epsilon)  := \{\, \tilde{\c} : \Vert \tilde{\c} - \c \Vert_\infty \leq \varepsilon\,\}$ and
consider the max-min problem
associated with $\P$:
\begin{equation}
\label{primalSDProb}
\rho\,=\,\displaystyle \max_{\tilde{\c} \in \B_\infty(\c,\varepsilon)} \:\inf_{\y} \,\{\: \tilde{\c}^T \, \y : 
\: \F(\y) \succeq 0\:\}.
\end{equation}
As in Section \ref{sec:two-player}, there is a simple two-player game interpretation of \eqref{primalSDProb}. Player 1 (the leader) searches for the ``best" cost function $\tilde{\c}\in\B_\infty(\c,\varepsilon)$ which is ``robust" against the {\em worst} decision $\y$ made 
by Player 2 (the follower, the decision maker), once Player 1's choice $\tilde{\c}$ is known. 

\begin{proposition}
\label{th:SDProb}
Assume that there exists $\hat{\y}$ such that $\F(\hat{\y})\succ0$. Then solving the max-min problem~\eqref{primalSDProb} is equivalent to solving :
\begin{equation}
\label{primalSDProb2}
\displaystyle \inf_{\y} \:\{\: \displaystyle \c^T \, \y + \varepsilon \, \|\y\|_1:\:
\F(\y) \succeq 0\:\}.
\end{equation}
\end{proposition}
\begin{proof}
$\F(\hat{\y})\succ0$ implies that Slater's condition holds for
the inner (minimization) SDP of \eqref{primalSDProb}. Therefore,
by standard conic duality:
\begin{equation}
\label{dualSDProb}
\rho=\displaystyle\max_{\tilde{\c}} \quad \sup_{\X\succeq0} \:\{\: \langle \F_0, \X\rangle:
\langle \F_j, \X \rangle = \tilde{c}_j \,, \quad \mid \tilde{c}_j - c_j  \mid \leq \varepsilon,\quad j=1,\ldots,n\:\},
\end{equation}
which in turn is equivalent to:
\begin{equation}
\label{dualSDProb2}
\displaystyle\sup_{\X\succeq0} \:\{ \: \langle \F_0, \X\rangle:\quad
\mid \langle \F_j, \X \rangle  - c_j  \mid \leq \varepsilon  \,, \quad j=1,\ldots,n\:\}.
\end{equation}
As in the proof of Proposition~\ref{th:primaldualepseta}, we prove that the dual of SDP~\eqref{dualSDProb2} is~\eqref{primalSDProb2}.
\end{proof}
So again, with an appropriate value of $\varepsilon$ related the 
the numerical precision of SDP solvers, \eqref{dualSDProb2} can be considered
as a fair model of treating inaccuracies by relaxing the equality constraints of \eqref{pb-sdp-dual}
up to some tolerance level $\varepsilon$.
That is, instead of solving exactly \eqref{pb-sdp-dual} with nominal criterion $\c$,
Player 1 (the SDP solver) is considering a related robust version where it solves (exactly) \eqref{pb-sdp-dual} but now with some optimal choice of a new cost vector $\tilde{\c}\in
\B_\infty(\c,\varepsilon)$. But this is a robustness point of view from the solver ({\em not} from the decision maker) and the resulting robust solution is some optimal cost vector $\tilde{\c}^*\in \B_\infty(\c,\varepsilon)$. 

In the particular case of SDP relaxations for  polynomial optimization, we retrieve~\eqref{primaljeps} as an instance of~\eqref{primalSDProb2} and~\eqref{dualjeps} as an instance of~\eqref{dualSDProb2}. 

\subsection*{Robust SDP} On the other hand,  the objective function $\tilde{\c}^T\y$ is bilinear in $(\tilde{\c},\y)$, 
the set $\B_\infty^j(\c,\varepsilon)$ is convex and compact, and
the set $\Y:=\{\y:\F(\y)\succeq0\}$ is convex. Hence
by Theorem \ref{th:sion}, \eqref{primalSDProb} is equivalent to solving the min-max problem:

\begin{equation}
\label{primalSDProb_minmax}
\rho\,=\,\displaystyle  \inf_{\y} \quad\{\: \max_{\tilde{\c} \in \B_\infty(\c,\varepsilon)}\: \{\: \tilde{\c}^T\y\:\} :
\quad \F(\y) \succeq 0\:\},
\end{equation}
which is a ``robust" version of \eqref{pb-sdp} from the point of view of the decision maker
when there is uncertainty in the cost vector. That is, the cost vector
$\tilde{\c}$ is not known exactly and belongs to the uncertainty set $\B_\infty(\c,\varepsilon)$.
The decision maker has to make a robust decision $\y^*$ with is the best against all possible 
values of the cost function $\tilde{\c}\in\B_\infty(\c,\varepsilon)$. This 
well-known latter point of view is that of {\em robust optimization} in presence of uncertainty for the cost vector; see e.g. \cite{RobustElGhaoui}.

So if the latter robustness point of view (of the decision maker) is well-known,
what is perhaps less known (but not so surprising) is that it can be interpreted in terms of a robustness point of view 
from an inexact ``solver" when treating equality constraints with inaccuracies in a 
problem with nominal criterion. Given problem \eqref{pb-sdp} with nominal criterion $\c$,
and without being asked to do so, the solver behaves {\em as if} it is solving {\em exactly}
the robust version \eqref{primalSDProb_minmax} (from the decision maker viewpoint), whereas the decision maker is willing to solve \eqref{pb-sdp} exactly. In other words, 
Sion's minimax theorem validates the informal (and not surprising) statement
that the treatment of inaccuracies by the SDP solver can be viewed as a robust treatment of uncertainties in the cost vector.

\if{
However, in the case of SDP relaxations for polynomial optimization, this behavior 
is indeed more surprising at first glance.
Indeed, some unconstrained optimization instances such as 
minimizing Motzkin-like polynomials (i.e., when $f-f^*$ is not SOS),  cannot be theoretically handled by SDP relaxations (assuming that one relies on exact SDP solvers). 
Yet, double floating point SDP solvers solve them in a practical manner, 
provided that  higher-order relaxations are allowed 
so that a polynomial of degree $d$ is treated as a higher degree polynomial 
(but with zeros coefficients for monomials of degree higher than $d$).
}\fi

However, in the case of SDP relaxations for polynomial optimization, this behavior
is indeed more surprising and even spectacular.
Indeed, some unconstrained optimization instances such as
minimizing Motzkin-like polynomials (i.e., when $f-f^*$ is not SOS),  cannot be theoretically handled by SDP relaxations (assuming that one relies on exact SDP solvers).
Yet, double floating point SDP solvers solve them in a practical manner,
provided that higher-order relaxations are allowed
so that a polynomial of degree $d$ can be (and indeed is!)  treated as a higher degree polynomial
(but with zero coefficients for monomials of degree higher than $d$).

In general, similar phenomena can occur while relying on general floating point algorithms.
We presume that they could also appear when handling polynomial optimization problems with alternative convex programming relaxations relying on interior-point algorithms, for instance linear/geometric programming. }

\if{
Indeed, for all $\tilde{\c} \in \B_\infty^j(\c,\varepsilon)$, one has $\mid \tilde{c}_\alpha - c_\alpha \mid \leq \varepsilon$, for all $\alpha \in \N^n_{2j}$, implying 
$\tilde{\c}^T \, \y - \c^T \, \y
 \leq 
 \mid \tilde{\c}^T \, \y - \c^T \, \y \mid \leq \varepsilon \, \Vert \y \Vert_1$, thus $\tilde{\c}^T\leq \c^T \, \y +\varepsilon \, \Vert \y \Vert_1$.\\
Therefore, the optimal value of the min-max problem~\eqref{primalSDProb_minmax} is less than the one of~\eqref{primalSDProb2}, which is equal to the one of the max-min problem~\eqref{primalSDProb}, by Proposition~\ref{th:SDProb}. Combining this fact together with the max-min inequality yields the equivalence between~\eqref{primalSDProb} and~\eqref{primalSDProb_minmax}.
} \fi
\section{Examples}
\label{sec:benchs}

All experimental results are obtained by computing the solutions of the primal-dual SDP relaxations~\eqref{primalj}-\eqref{dualj} of Problem~$\P$. These SDP relaxations are implemented in the \href{https://gricad-gitlab.univ-grenoble-alpes.fr/magronv/RealCertify}{$\realcertify$}~\cite{realcertify} library, available within {\sc Maple}, and interfaced with the SDP solvers SDPA~\cite{Yamashita10SDPA} and SDPA-GMP~\cite{Nakata10GMP}. 

For the two upcoming examples, we rely on the procedure described in~\cite{Henrion2005} to extract the approximate global minimizer(s) of some given objective polynomial functions.
We compare the results obtained with (1) the SDPA solver implemented in double floating point precision, which corresponds to $\epsilon = 10^{-7}$ and (2) the arbitrary-precision SDPA-GMP solver, with $\epsilon = 10^{-30}$. The value of our robust-noise model parameter $\varepsilon$  roughly matches with the one of the parameter \texttt{epsilonStar} of SDPA.

We also noticed that decreasing the value of the SDPA parameter \texttt{lambdaStar} seems to boil down to increasing the value of our robust-noise model parameter $\eta$. 
An expected justification is that \texttt{lambdaStar} is used to determine a starting point $\X^0$ for the interior-point method, i.e., such that $\X^0 = \mathtt{lambdaStar} \times \I$ (the default value of \texttt{lambdaStar} is  equal to $10^{2}$ in SDPA and is equal to $10^{4}$ in SDPA-GMP). 
A similar behavior occurs when decreasing the value of the parameter \texttt{betaBar}, which controls the search direction of the interior-point method when the matrix $\X$ is not positive semidefinite. 

However, the correlation between the values of \texttt{lambdaStar} (resp.~\texttt{betaBar}) and $\eta$ appears to be nontrivial. Thus, our  robust-noise model would be theoretically valid if one could impose the value of a parameter $\eta$, ensuring that $\X \succeq -\eta \: \I$ when the interior-point method terminates. From the best of our knowledge, this feature happens to be unavailable in modern SDP solvers.
For that reason, our experimental comparisons are performed by changing the value of \texttt{epsilonStar} in the parameter file of the SDP solver.


\subsection{Motzkin polynomial}

Here, we consider the Motkzin polynomial $f = \frac{1}{27} + x_1^2 x_2^2 (x_1^2 + x_2^2 - 1)$. This polynomial is nonnegative but is not SOS. The minimum $f^\star$ of $f$ is 0 and $f$ has four global minimizers with coordinates $x_1 = \pm \frac{\sqrt{3}}{3}$ and $x_2 = \pm \frac{\sqrt{3}}{3}$. As noticed in~\cite[Section 4]{Henrion2005}, one can retrieve these global minimizers by solving 
the primal-dual SDP relaxations~\eqref{primalj}-\eqref{dualj} of Problem~$\P$ at relaxation order $j = 8$:
\begin{itemize}
\item[(1)] With  $\epsilon = 10^{-7}$, we obtain an approximate lower bound of $-1.81 \cdot 10^{-4} \leq f^\star$, as well as the four global minimizers of $f$ with the extraction procedure. The dual SDP~\eqref{dualj} allows to retrieve the approximate SOS decomposition $f(\x) = \sigma(\x) + r(\x)$, where $\sigma$ is an SOS polynomial and the corresponding polynomial remainder $r$ has coefficients of approximately equal magnitude, and which is less than $10^{-8}$.
\item[(2)] With $\epsilon = 10^{-30}$, we obtain an approximate lower bound of $-1.83 \cdot 10^{1} \leq f^\star$ and the extraction procedure fails. The corresponding polynomial remainder has coefficients of magnitude less than $10^{-31}$.
\end{itemize}

We notice that the support of $r$ contains only terms of even degrees, i.e., terms of the form $\x^{2 \beta}$, with $|\beta| \leq 8$.
Hence we consider a perturbation $\tilde{f}_{\gamma}$ of $f$ defined by $\tilde{f}_{\gamma}(\x) = f(\x) + \gamma \sum_{|\beta| \leq j} \x^{2 \beta}$,  with $\gamma = 10^{-8}$. 
By solving the SDP relaxation (with $j=8$) associated to $\tilde{f}_{\gamma}$, with $\varepsilon = 10^{-30}$, we retrieve again the four global minimizers of $f$.

\subsection{Univariate polynomial with minimizers of different magnitudes}
~\\
We start by considering the following univariate optimization problem:
\[ f^\star = \min_{ x \in \R } f(x) \,,
\] 
with $f(x) = (x-100)^2 \Bigl((x - 1)^2 + \frac{\gamma}{99^2} \Bigr)$ and $\gamma \geq 0$.
~\\
Note that the minimum of $f$ is $f^\star = 0 = f(100)$ and $f(1) = \gamma$.


We first examine the case where $\gamma = 0$. In this case, $f$ has two global minimizers $1$ and $100$. 
At relaxation order $j$, with $2 \leq j \leq 5$, we retrieve the following results (rounded to four significant digits):
\begin{itemize}
\item[(1)] With $\epsilon = 10^{-7}$, we obtain  $\hat{x}^{(1)} = 0.9999 \simeq 1$, corresponding to the smallest global minimizer of $f$.
\item[(2)] With $\epsilon = 10^{-30}$, we obtain $\hat{x} = 50.5000 = \frac{1 + 100}{2}$, corresponding to the average of the two global minimizers of $f$.
\end{itemize}

We also used the \texttt{realroot} procedure, available within Maple, to compute the local minimizers of the following function on $[0, \infty)$: 
\begin{align}
\label{eq:tildef}
\tilde{f}_{\varepsilon,j}(x) = f(x) +   \varepsilon \sum_{|\alpha| \leq 2 j} | x^\alpha | = f(x) +   \varepsilon \sum_{|\alpha| \leq 2 j}  x^\alpha  \,, 
\end{align} 
\begin{itemize}
\item[(1)] With $\epsilon = 10^{-7}$, we obtain $\tilde{x}^{(1)} = 0.9961 \simeq \hat{x}^{(1)}$.
\item[(2)] With  $\epsilon = 10^{-30}$, we obtain $\tilde{x}^{(1)} = 0.9961 \simeq \hat{x}^{(1)}$ and $\tilde{x}^{(2)} =  99.9960 \simeq 100$, the largest global minimizer of $f$. The corresponding values of $\tilde{f}_{\varepsilon,j}$ are $0.1496$ and $0.1495$, respectively.
\end{itemize}
These experiments confirm our explanations that the solver computes the solution of SDP relaxations associated to the perturbed function $\tilde{f}_{\varepsilon,j}$ from~\eqref{eq:tildef}. With double floating point precision (1), this perturbed function has a single minimizer, retrieved by the extraction procedure. With higher precision (2), this perturbed function has two local minimizers, whose average is retrieved by the extraction procedure.

Next, we examine the case where $\gamma = 10^{-3}$. In this case, $f$ has a single global minimizer, equal to $100$ and another local minimizer 
At relaxation order $j$, with $2 \leq j \leq 5$, we retrieve the following results (rounded to four significant digits):
\begin{itemize}
\item[(1)] With $\epsilon = 10^{-7}$, we obtain  $\hat{x}^{(1)} = 0.9999 \simeq 1$, corresponding to the smallest global minimizer of $f$ when $\gamma = 0$.
\item[(2)] With $\epsilon = 10^{-30}$, we obtain $\hat{x}^{(2)} = 99.1593 \simeq 100$, corresponding to the single global minimizer of $f$.
\end{itemize}

We also compute the local minimizers of $\tilde{f}_{\varepsilon,j}$ with \texttt{realroot}:
\begin{itemize}
\item[(1)] With $\epsilon = 10^{-7}$, we obtain $\tilde{x}^{(1)} = 1.0039 \simeq \hat{x}^{(1)}$.
\item[(2)] With $\epsilon = 10^{-30}$, we obtain $\tilde{x}^{(1)} = 1.0039 \simeq \hat{x}^{(1)}$ and $\tilde{x}^{(2)} =  99.9961 \simeq 100$, the single global minimizer of $f$. The corresponding values of $\tilde{f}_{\varepsilon,j}$ are $0.1505$ and $0.1495$, respectively. This confirms that $\tilde{x}^{(2)}$ is the single global minimizer of $\tilde{f}_{\varepsilon,j}$, approximately extracted, as $\hat{x}^{(2)}$.
\end{itemize}

Here again, our robust-noise model, relying on the perturbed polynomial function $\tilde{f}_{\varepsilon,j}$, fits with the above experimental observations. This perturbed function has a single global minimizer, whose value depends on the parameter $\varepsilon$, and which can be approximately retrieved by the extraction procedure.

\section{Discussion}
By considering the hierarchy of SDP relaxations associated to a given polynomial optimization problem, we are facing with a dilemma when relying on numerical SDP solvers. On the one hand, we might want to increase the precision of the solver to get rid of the numerical uncertainties and obtain an accurate solution of the SDP relaxations. On the other hand, working with low precision may allow to obtain hints related to the solution of the initial problem. This has already happened in both commutative and non-commutative contexts, to compute the global minimizers of the Motzkin polynomial in~\cite{Henrion2005} or the bosonic energy levels from~\cite{Navascues13}. 
Our theoretical robust-noise model could be extended to problems addressed with structured SDP programs (as, for instance, the moment and localizing matrices coming from polynomial optimization problems). We believe that the use of ``inaccurate'' SDP solvers could also provide hints for the solutions of such problems.
{One could estimate how close are the optimal values of duals~\eqref{primaljeta} and~\eqref{primaljeps} of noise models, to the optimal values of the initial optimization problem $\P$. For some instances in this article and other papers from the literature, the optimal values seem to not exceed the optimal value of $\P$ for higher orders sufficiently large. Such experimental observations
remain to be explained and/or validated.}

\bibliographystyle{siamplain}

\end{document}